\documentclass[11pt,letterpaper,reqno]{amsart}
\usepackage{amsmath}
\usepackage{amstext}
\usepackage{amssymb}
\usepackage{amsfonts}
\usepackage{enumerate}
\usepackage{mathrsfs}
\usepackage{color}
\usepackage[all]{xy}

\numberwithin{equation}{section}

\newtheoremstyle{note}
{1em}
{1em}
{}
{}
{\bfseries}
{:}
{.5em}
{}

\newtheorem{theorem}{Theorem}[section]
\newtheorem{lemma}[theorem]{Lemma}
\newtheorem{proposition}[theorem]{Proposition}
\newtheorem{corollary}[theorem]{Corollary}
\newtheorem{definition}[theorem]{Definition}

\theoremstyle{note}
\newtheorem{remark}[theorem]{Remark}

\newcommand{\lap}{{\Delta}}
\newcommand{\N}{{\mathbb{N}}}

\newcommand{\R}{{\mathbb{R}}}
\newcommand{\C}{{\mathbb{C}}}

\newcommand{\pair}[2]{{\langle #1, #2 \rangle}}

\renewcommand{\Re}{\operatorname{Re}}

\DeclareMathOperator{\vol}{Vol}

\author[J.A.~Ch\'avez-Dom\'{\i}nguez]{Javier Alejandro Ch\'avez-Dom\'{\i}nguez}
\address{Department of Mathematics,
University of Oklahoma,
Norman OK , 73019-3103 USA}

\email{jachavezd@ou.edu}
\thanks{The author was partially supported by NSF grants DMS-1400588 and DMS-1900985.}


\begin{document}

\title[Isoperimetric and Sobolev inequalities for magnetic graphs]{Isoperimetric and Sobolev inequalities for magnetic graphs}

\begin{abstract}
We introduce a concept of isoperimetric dimension for magnetic graphs, that is, graphs where every edge is assigned a complex number of modulus one.
In analogy with the classical case, we show that isoperimetric inequalities imply Sobolev inequalities on such graphs.
As a first application, we show that the signed Cheeger constant behaves additively with respect to Cartesian products of graphs.
Using heat kernel techniques, we also give lower bounds for the eigenvalues of the discrete magnetic Laplacian.
\end{abstract}

\maketitle

\section{Introduction}

Many practical situations can be modeled using a (combinatorial) graph, a set of vertices where some pairs of them are joined by an edge.
The vertices correspond to the different objects of interest, and the presence of an edge indicates that two particular objects are related.
In some cases, we have more information available: not only that two objects \emph{are} related, but also \emph{how} they are related.
We will particularly look at the case where this information is given by a complex number of modulus one, giving rise to what is known as \emph{magnetic graphs}.
For example, the angular synchronization problem \cite{Singer} is to estimate $n$ unknown angles $\theta_1, \dotsc, \theta_n \in [0,2\pi)$ from a set of measurements of their offsets $\theta_k -\theta_j \mod 2\pi$;
in general, only a subset of all possible offsets are measured. The set $E$ of pairs $\{k, j\}$
for which offset measurements exist can be realized as the edge set of a graph $G = (V , E)$ with vertices corresponding to the 
angles, and edges corresponding to measurements. Each measurement can be encoded as a complex number $\sigma_{kj} = e^{i(\theta_k-\theta_j)}$ having modulus one, and observe that $\sigma_{kj} = \sigma_{jk}^{-1}$.
Another such situation is quantum mechanics on a graph, as in \cite{Lieb-Loss}.
The vertices of the graph can be thought of as locations of atoms in a solid,
and the edges correspond to electron bonds joining
the atoms; this model is known as the tight-binding model or H\"uckel model.
The argument of the complex number associated to an edge joining $u$ and $v$ has a physical interpretation, as the
integral of a magnetic vector potential from the point $u$ to the point $v$.
Since reversing the direction of integration results in a change of sign, reversing the direction of travel along an edge once again replaces the associated complex number by its multiplicative inverse. See Section \ref{sec:basic-notation} for the mathematical formalization of these ideas.

The main goal of this paper is to develop, in the context of magnetic graphs, a subject with a long history: the relationship between isoperimetric and Sobolev inequalities.
This is well known for domains in $\R^n$, and the fact that they are equivalent was shown by Maz'ya \cite{Mazya}.
The fundamental tool giving this connection is the coarea formula, a version of Cavalieri's principle which states that the integral of the length of the gradient of a function $g$ can be calculated as the integral of the measures of the level sets of $g$; see \cite[Sec. 2]{Mazya-Lec} for a nice elementary exposition.
In the graph case, Sobolev inequalities have also been obtained from isoperimetric ones through a discrete version of the coarea formula \cite{ChungYau95,Tillich,Ostrovskii-Sobolev}.
Various applications of Sobolev inequalities for classical graphs are given, for example, in \cite{Chung-Tetali} and \cite{ChungYau95}, and we prove versions of these results for magnetic graphs below.
As can be expected from the discussion above, our fundamental tool will be a version of the coarea formula in this context: we obtain it as a variation of \cite[Lemma 4.3]{LLPP}.

The rest of this paper is organized as follows.
In Section \ref{sec:basic-notation} we introduce the basic notation and notions associated to magnetic graphs.
Section \ref{sec:isoperimetry} introduces the notion of isoperimetric constants for magnetic graphs, and proves that they imply associated Sobolev inequalities.
Applications are given in the next two sections:
Section \ref{sec:products} shows that the signed Cheeger constant behaves additively with respect to Cartesian products of graphs,
whereas Section \ref{sec:heat} uses heat kernel techniques to give lower bounds for the eigenvalues of the discrete magnetic Laplacian.
The last two sections are brief remarks:
Section \ref{sec:S1-valued} indicates how to generalize the main results from Section \ref{sec:isoperimetry}, and Section \ref{sec:BL} relates our Sobolev inequalities with Balian-Low type theorems in finite dimensions from time-frequency analysis.

\section{Basic notation}\label{sec:basic-notation}

With the exception of the last proposition, the notation and definitions in this section follow \cite{LLPP}.
We always denote by $G=(V,E)$ an undirected simple finite
graph on $N$ vertices, with vertex set $V$ and edge set $E$.
We consider the edges to be unordered pairs
 $\{u,v\}$, and use the notation $u \sim v$ to indicate that $u \in V$ and $v
\in V$ are connected by an edge.
Sometimes we will need to consider oriented edges, which we denote by ordered pairs.
If $e=(u,v)$ is the oriented edge starting at $u$ and terminating at $v$, we write $\bar
e=(v,u)$ for the oriented edge with the reversed orientation.
The set of all oriented edges will be denoted by
$E^{or}:=\{(u,v), (v,u)\mid \{u,v\}\in E\}$.
To every edge $e=\{u,v\} \in E$ we associate a positive symmetric weight $w_{uv} = w_{vu} = w_e$,
and we define the weighted degree $d_u$ of a vertex $u\in V$ to be
$d_u:=\sum_{v,v\sim u}w_{uv}$.
Given a positive measure $\mu: V \to
{\R}^+$ on $V$, the \emph{maximal $\mu$-degree} of the graph
$G$ is
$$
  d_{\mu}:=\max_{u\in V}\left\{\frac{\sum_{v,v\sim u}w_{uv}}{\mu(u)}\right\}
  = \max_{u\in V}\left\{\frac{d_u}{\mu(u)}\right\}.
$$

We denote the \emph{boundary measure} of $X \subseteq V$ by
$$
|E(X,X^c)|:=\sum_{u\in X}\sum_{v\in X^c}w_{uv},
$$
where $X^c$ is the complement of $X$ in $V$.
The \emph{$\mu$-volume} of
$X$ is given by
$$
\vol_{\mu}(X):=\sum_{u\in X}\mu(u).
$$
In a slight abuse of notation, in the special case $X=V$ we will often write $\vol_\mu(G)$ instead of $\vol_\mu(V)$.

Given $k \in {\mathbb N}$, we use the standard combinatorial notation
$[k] = \{1,2,\dots,k\}$.
Throughout the paper, we will only consider the case
where the signature group $\Gamma$ is either the unit circle $S^1 = \{ z \in \C \mid |z|=1 \}$, or the cyclic group
$S_k^1:=\{\xi^{j} \mid j \in [k] \}$ of order $k$ generated by
the primitive $k$-th root of unity $\xi:=e^{2\pi i/k}\in \C$.
Below, whenever we consider a group $\Gamma$ it will always be either $S^1$ or $S^1_k$.

\begin{definition} Let $G$ be a graph and let $\Gamma \subseteq S^1$ be a group. A \emph{signature} for
  $G$ is a map $s: E^{or} \to \Gamma$ such that
  \begin{equation}\label{eq:signatureKEY}
    s(\bar e)=s(e)^{-1},
  \end{equation}
  where $s(e)^{-1}$ is the inverse of $s(e)$ in $\Gamma$. The \emph{trivial
  signature} $s \equiv 1$, is denoted by $s_1$. For an oriented edge
  $e=(u,v) \in E^{or}$, we will almost always write $s_{uv}:=s(e)$ for
  simplicity.
\end{definition}

Given a $\Gamma$-valued signature on $G$, and a $\Gamma$-valued function on the vertices of $G$,
it is easy to produce a new signature by ``conjugating'' the old one by the function on the vertices.  
This operation is called \emph{switching}.

\begin{definition}\label{def:switching} Let $G$ be a graph with $\Gamma$-valued signature $s$. For any function
  $\tau: V\to \Gamma$ we can define a new signature
  $s^{\tau}:E^{or}\to \Gamma$ as follows:
$$
    s^{\tau}(e)=\tau(u)s(e)\tau(v)^{-1} \quad \forall\, e=(u,v)\in E^{or}.
$$
  We say that the function $\tau$ is a \emph{switching function}.  Two signatures $s$
  and $s'$ are said to be \emph{switching equivalent} if there exists a
  switching function $\tau$ such that $s'=s^{\tau}$.
\end{definition}

We remark that it is easy to check that switching is in fact an equivalence relation on the
set of signatures.

\subsection{Frustration index}

Let $G$ be a finite graph with a
signature $s: E^{or} \to \Gamma\subseteq S^1$, and let
$\mathcal{C}$ be a cycle, which is a sequence $(u_1,u_2), (u_2,u_3), \cdots, (u_{l-1}, u_l), (u_l,u_1)$ of distinct edges in $E$.
Then the signature of $\mathcal{C}$ is
\begin{equation*}
s(\mathcal{C}) := s_{u_1u_2}s_{u_2,u_3}\cdots
s_{u_{l-1}u_l}s_{u_lu_1}\in \Gamma.
\end{equation*}
Note that the signature of a
cycle is switching invariant.

\begin{definition}
  A signature $s:E^{or}\to \Gamma$ is called \emph{balanced} if the
  signature of every cycle of $G$ is equal to 1.
\end{definition}

We will also say that the graph $G$ or a subgraph of
$G$ is balanced if the signature restricted to it is balanced.
Since the signature of a cycle is switching invariant, the property of being
balanced is also switching invariant. The following is an important
characterization of being balanced using switching operations \cite[Corollary
    3.3]{Zaslavsky82}.

\begin{proposition}\label{Pro:switching lemma}
A signature $s: E^{or}\to \Gamma\subseteq S^1$ is balanced if and only if it is switching equivalent to the trivial signature $s_{1}$.
\end{proposition}

In fact, more is true \cite[Lemma 2.1]{Lieb-Loss}: if two signatures take the same values on all cycles, then the signatures are switching equivalent.  

If a signature is not balanced on a subset, we would like to quantify how far it is from being balanced.
For that purpose,
the following frustration index was defined in \cite[Defn. 3.4]{LLPP}.

\begin{definition}\label{def:frustration index}
  Let $G$ be a finite graph with $\Gamma$-valued signature $s$ and $V_1\subseteq V$ a nonempty subset
  with induced subgraph~$(V_1,E_1)$. The \emph{frustration index} $\iota^s(V_1)$
  of $V_1$ is defined as
  \begin{align}\label{eq:frustration index}
    \iota^s(V_1):&=\min_{\tau: V_1\to
      \Gamma}\sum_{\{u,v\}\in E_1}w_{uv}|\tau(u)-s_{uv}\tau(v)| \\
      &=\min_{\tau: V_1\to \Gamma}\sum_{\{u,v\}\in E_1}w_{uv}|1-\tau(u)^{-1}s_{uv}\tau(v)| \notag
  \end{align}
\end{definition}

It is clear from the definition that the frustration index of a set is switching invariant, and also that
 the frustration index is a measure, in an $\ell_1$ sense, of how close we can get to the trivial signature using the switching operation.
In particular, according to Proposition \ref{Pro:switching
 lemma}, we have that
  $\iota^s(V_1)=0$ if and only if  the subgraph induced by  $V_1$  is balanced.

Note that the frustration index of any signature on a tree is always zero, since one can easily construct inductively a switching function that gives equivalence with the trivial signature.
 As an example we now calculate the frustration index of a signature on a cycle with unit weights.

\begin{proposition}\label{prop:frustration-of-cycle}
Let $\mathcal{C}=(V,E)$ be a graph which is a cycle consisting of the sequence of distinct edges 
$(u_1,u_2), (u_2,u_3), \cdots, (u_{n-1}, u_n), (u_n,u_1)$, each having weight $1$, and let $s$ be an $S^1_k$-valued (or $S^1$-valued) signature for $\mathcal{C}$. Then
$$
\iota^s(\mathcal{C}) = \big| 1 - s(\mathcal{C}) \big|.
$$
\end{proposition}

\begin{proof}
For simplicity let us denote $\sigma_j = s_{u_ju_{j+1}}$, with the convention $u_{n+1}= u_1$.
Define $\tau : V \to S^1_k$ by $\tau(u_1) = 1$, and for $j>1$ let $\tau(u_j) = \sigma_{j-1}^{-1} \cdots \sigma_1^{-1}$.
Then
\begin{multline*}
\sum_{j=1}^n |\tau(u_j) - \sigma_j\tau(u_{j+1})| \\= |\tau(u_n) - \sigma_n\tau(u_1)| = \big| \sigma_{n}^{-1} \cdots \sigma_1^{-1} - 1 \big| = \big| 1 - \sigma_1 \cdots \sigma_n \big|.
\end{multline*}
On the other hand, for any $\tau : V \to S^1_k$, by the triangle inequality,
\begin{multline*}
\sum_{j=1}^n \big|\tau(u_j) - \sigma_j\tau(u_{j+1})\big| 
= \sum_{j=1}^n \big|\sigma_1 \cdots \sigma_{j-1}\tau(u_j) - \sigma_1 \cdots \sigma_{j-1}\sigma_j\tau(u_{j+1})\big| \\
\ge \big| \tau(u_1) - \sigma_1 \cdots \sigma_n \tau(u_1) \big| = \big| 1 - \sigma_1 \cdots \sigma_n \big|.
\end{multline*}
\end{proof}

\section{From isoperimetric to Sobolev inequalities}\label{sec:isoperimetry}

The concept of isoperimetric dimension of a graph was introduced in \cite{ChungYau95}, as a ``finite-dimensional'' variation of the famous Cheeger constant,
and it was used to proved Sobolev inequalities for graphs.
The main point of this paper is to do an analogous study for graphs with unbalanced  signatures,
 building upon the definition of Cheeger constants in this context from \cite{LLPP}.

\begin{definition}\label{def:isoperimetric constant}
  Let $G$ be a finite graph with a signature $s$. 
  We say that $(G,s)$ has \emph{isoperimetric dimension $\delta$} with \emph{isoperimetric constant} $c_\delta$
if for every nonempty subset $V_1$ of $V$ we have
  \begin{equation}\label{eq:isoperimetric constant}
    \iota^s(V_1)+|E(V_1, V_1^c)| \ge c_\delta {\vol_{\mu}(V_1)} ^{\frac{\delta-1}{\delta}}.
  \end{equation}
\end{definition}

Note that the case $\delta=\infty$ is precisely the ($1$-way) Cheeger constant from \cite[Def. 3.5]{LLPP}.
The reader should be warned that this does not directly correspond to the classical Cheeger constant for (non-magnetic) graphs;
that would be the $2$-way Cheeger constant.

\subsection{The coarea inequality}

This subsection is an adaptation of \cite[Lemmas 4.2 and 4.3]{LLPP}.
We start by setting up the notation involved.
Let $B_r(0):=\{z\in \mathbb{C}\mid |z|< r\}$ be the open disk in $\C$ with center $0$ and radius
$r$.
For $\theta\in [0, 2\pi)$ and $k \in {\N}$, we define the
following $k$ disjoint sectorial regions
\begin{equation}\label{eq:notationQ}
  Q_{j}^{\theta}:=\left\{re^{i\alpha}\in \overline{B_1(0)}\left|\,\, r\in (0,1], \alpha\in \left[\theta+\frac{2\pi j}{k}, \theta+\frac{2\pi (j+1)}{k}\right)\right.\right\},
\end{equation}
where $j=0,1,\ldots, k-1$. Then for any $t\in (0,1]$, we define the function
$Y_{t, \theta}: \overline{B_1(0)}\to \mathbb{C}$ as
\begin{equation}
  Y_{t, \theta}(z):=\left\{
    \begin{array}{ll}
      \xi^j, & \hbox{if $z\in Q_j^{\theta}\setminus B_{t}(0)$,} \\
      0, & \hbox{if $z\in B_{t}(0)$,}
    \end{array}
  \right.
\end{equation}
where $\xi$ denotes the $k$-th primitive root of unity.

\smallskip

The following lemma is a variation of \cite[Lemma 4.2]{LLPP}.

\begin{lemma}\label{lemma:key}
  For any two points $z_1, z_2\in \overline{B_1(0)}$, we have
 $$
    \frac{1}{2\pi}\int_0^{2\pi}\int_0^1 \left| Y_{t, \theta}(z_1)-
      Y_{t, \theta}(z_2) \right|\, dt\, d\theta\, \leq \,
    3\, |z_1-z_2|.
$$
\end{lemma}

\begin{proof}
Without loss of generality, assume that $|z_1| \ge |z_2|$ with
  $z_1\in Q_{j_1}^{\theta}$ and $z_2\in Q_{j_2}^{\theta}$. Then we
  have
$$
    |Y_{t, \theta}(z_1)-Y_{t, \theta}(z_2)|=\left\{
      \begin{array}{ll}
        |\xi^{j_1}-\xi^{j_2}|, & \hbox{if $t\leq |z_2|$,} \\
        1, & \hbox{if $|z_2|<t\leq |z_1|$,} \\
        0, & \hbox{if $|z_1|<t$.}
      \end{array}
    \right.
$$
  Hence,
$$
    \int_0^1\left| Y_{t, \theta}(z_1)-Y_{t, \theta}(z_2)\right|\, dt=|\xi^{j_1}-\xi^{j_2}|\cdot|z_2|+(|z_1|-|z_2|).
$$
  Let $\alpha\in [0, \pi]$ be the angle between the two rays
  joining $z_1, z_2$ to the origin. If $2\pi l/k\leq \alpha<
  2\pi (l+1)/k$ for some integer $0\leq l<k/2$, the term
  $|\xi^{j_1}-\xi^{j_2}|$ is equal to either $|1-\xi^l|$ or
  $|1-\xi^{l+1}|$, hence we calculate
  \begin{equation*}
  \frac{1}{2\pi}\int_0^{2\pi}\int_0^1 \left| Y_{t,
        \theta}(z_1)- Y_{t, \theta}(z_2) \right|\, dt\, d\theta =
  \phantom{\frac{k\left(\alpha_{z_1z_2}-2\pi l/k\right)}{2\pi}
    \left(|1-\xi^{l+1}|\cdot|z_2|^2 + |z_1|^2 - |z_2|^2 \right)}
  \end{equation*}
  \begin{eqnarray}
\notag  &=& \left(\frac{k\alpha}{2\pi}-l\right)  \left(|1-\xi^{l+1}|\cdot|z_2| + |z_1| - |z_2| \right)\\
\notag  && +\left(l+1-\frac{k\alpha}{2\pi}\right)  \left(|1-\xi^{l}|\cdot|z_2|+|z_1|-|z_2|\right)\\
\notag  &=& \left( \frac{k\alpha}{2\pi}-l\right)  \left(  |1-\xi^{l+1}| - |1-\xi^{l}|  \right)   \cdot|z_2| +  |1-\xi^l|\cdot|z_2|\, +\, \left(|z_1|-|z_2|\right) \\
\label{lemma:key-upper-bound}
  \end{eqnarray}
 At this point we consider two cases:
 \begin{enumerate}[(a)]
 
 \item If $l > 0$: Then we have $|1-\xi^{l+1}|\leq |1-\xi|+|1-\xi^l|\leq 2|1-\xi^l|$, and since $0 \le \frac{k\alpha}{2\pi} - l \le 1$ the quantity at the end of \eqref{lemma:key-upper-bound} is bounded above by
 $$
 2|z_2|\cdot|1-\xi^l| + |z_1 - z_2|
 $$
   Observe that we have
 $$
    |z_1-z_2|\geq \left|\frac{z_1}{|z_1|}|z_2|-z_2\right| = |z_2| \cdot |1 - e^{i\alpha}| \geq  |z_2|\cdot|1-\xi^l|
 $$
from where we get the desired bound of $3 |z_1-z_2|$.
 
 \item If $l = 0$: The quantity at the end of \eqref{lemma:key-upper-bound} reduces to
 $$
 \frac{k\alpha}{2 \pi} \cdot |1-\xi| \cdot |z_2| + \left(|z_1|-|z_2|\right) \le  \frac{k\alpha}{2 \pi} \cdot |1-\xi| \cdot |z_2| + |z_1-z_2|.
 $$
 Notice that
 \begin{multline*}
  \frac{k\alpha}{2 \pi} \cdot |1-\xi| \cdot |z_2| =  \frac{k\alpha}{2 \pi} \cdot |1-e^{2\pi i / k}| \cdot |z_2|  \\
  \le   \frac{k\alpha}{2 \pi}   \frac{2\pi}{k} \cdot |z_2| = \alpha \cdot |z_2| \le  |z_2| \cdot |1 - e^{i\alpha}|
 \end{multline*}
 which is bounded above by $|z_1-z_2|$ as calculated in the previous part.
 \end{enumerate}
 
   In either case, we obtain
$$
    \frac{1}{2\pi}\int_0^{2\pi}\int_0^1 \left|Y_{t, \theta}(z_1)-
    Y_{t, \theta}(z_2)\right|\, dt\, d\theta \leq
    3|z_1-z_2|.
$$
as desired.
\end{proof}

The previous lemma can be understood as a version of a coarea inequality,
which becomes clearer from the following variation of \cite[Lemma 4.3]{LLPP}.

\begin{lemma}[Coarea inequality]\label{lemma:coarea}
  Let $s: E^{or}\to S_k^1$ be a signature of $G$. For any
  function $f:V\to \C$ with $\max_{u\in V}|f(u)|=1$, we have
  \begin{multline*} 
    \int_0^1 \iota^s\big(\{|f|\ge t\}\big)\, + \, \big|E\big(\{|f|\ge t\}, \{|f|\ge t\}^c \big)\big| \, dt
     \\ \leq 3\sum_{\{u,v\}\in E}w_{uv}\, \left|f(u)-s_{uv}f(v)\right|.
  \end{multline*}
\end{lemma}

\begin{proof}
  First observe that
  \begin{multline*}
    \frac{1}{2\pi}\int_0^{2\pi}\int_0^1\sum_{\{u,v\}\in E}w_{uv}\,
    \left| Y_{t, \theta}(f(u))-
      s_{uv}Y_{t, \theta}(f(v))\right|\, dt\, d\theta\\
    \geq  \int_0^1 \iota^s\big(\{|f|\ge t\}\big)\, + \, \big|E\big(\{|f|\ge t\}, \{|f|\ge t\}^c \big)\big| \, dt. 
  \end{multline*}
In fact, the summation in the integrand of the LHS of the above inequality can be split into two parts:
\begin{enumerate}[(a)]
\item The summation over edges connecting one vertex from $\{|f|\ge t\}$ and one from $\{|f|\ge t\}^c$: This part equals  $\big|E\big(\{|f|\ge t\}, \{|f|\ge t\}^c \big)\big|$.
\item The summation over edges connecting two vertices from $\{|f|\ge t\}$: This part is bounded from below by $\iota^s\left(\{|f|\ge t\}\right)$ by Definition \ref{def:frustration index}.
\end{enumerate}

Notice further that
$$
    s_{uv}Y_{t,\theta}(f(v))=Y_{t,\theta}(s_{uv}f(v)),
$$
the desired inequality
now follows directly from Lemma \ref{lemma:key}.
\end{proof}

\subsection{Sobolev inequality: case $p=1$}

We are now ready to prove our first Sobolev inequality.

\begin{theorem}\label{thm:Sobolev-p=1}
Let $G$ be a finite graph with an $S_k^1$-valued signature $s$ such that $(G,s)$ has isoperimetric dimension $\delta$ with isoperimetric constant $c_\delta$.
For any function $f : V \to \C$ we have
$$
\left[  \sum_{u \in V}  \big|  f(u)  \big|^{\frac{\delta}{\delta-1}}  \mu(u) \right]^{\frac{\delta-1}{\delta}} \le \frac{3}{c_\delta} \sum_{\{u,v\}\in E}w_{uv}\, \left|f(u)-s_{uv}f(v)\right|.
$$
\end{theorem}

\begin{proof}
Assume without loss of generality that the function $f$ is not identically zero, otherwise there is nothing to prove.
After normalizing, we may assume that $\max_{u\in V}|f(u)|=1$.
From Lemma \ref{lemma:coarea} and Definition \ref{def:isoperimetric constant} we immediately get
$$
3 \sum_{\{u,v\}\in E}w_{uv}\, \left|f(u)-s_{uv}f(v)\right| \ge c_\delta \int_0^1  {\vol_{\mu}\big(\{|f|\ge t\}\big)} ^{\frac{\delta-1}{\delta}}  dt
$$
Now, using Minkowski's integral inequality
\begin{align*}
 \int_0^1  {\vol_{\mu}\big(\{|f|\ge t\}\big)} ^{\frac{\delta-1}{\delta}}  dt &= \int_0^1 \bigg[  \sum_{u \in V} \chi_{\{|f|\ge t\}}(u) \mu(u) \bigg]^{\frac{\delta-1}{\delta}} dt \\
 &= \int_0^1 \bigg[  \sum_{u \in V} \big| \chi_{\{|f|\ge t\}}(u) \big|^{\frac{\delta}{\delta-1}} \mu(u) \bigg]^{\frac{\delta-1}{\delta}} dt \\
 &\ge \left[  \sum_{u \in V}  \bigg|  \int_0^1 \chi_{\{|f|\ge t\}}(u) dt  \bigg|^{\frac{\delta}{\delta-1}}  \mu(u) \right]^{\frac{\delta-1}{\delta}} \\
 &= \left[  \sum_{u \in V}  \big|  f(u)  \big|^{\frac{\delta}{\delta-1}}  \mu(u) \right]^{\frac{\delta-1}{\delta}}
\end{align*}
from where the conclusion follows.
\end{proof}

The constant $3/c_\delta$ in Theorem \ref{thm:Sobolev-p=1} is essentially optimal, as illustrated by the following corollary. Compare to \cite[Cor. 1]{Chung-Tetali}.

\begin{corollary}\label{cor:variational-characterization-isoperimetric-constant}
Let $G$ be a finite graph with an  $S_k^1$-valued signature $s$ such that $(G,s)$ has isoperimetric dimension $\delta$ with isoperimetric constant $c_\delta$.
Then
$$
\frac{c_\delta}{3} \le \inf_{f \not= 0} \left[ \frac{ \sum_{\{u,v\}\in E}w_{uv}\, \left|f(u)-s_{uv}f(v)\right| }{ \left[ \sum_{u \in V}  \big|  f(u)  \big|^{\frac{\delta}{\delta-1}} \mu(u) \right]^{\frac{\delta-1}{\delta}}}
\right] \le c_\delta
$$
\end{corollary}

\begin{proof}
The lower bound follows immediately from Theorem \ref{thm:Sobolev-p=1}.
To prove the upper bound it suffices to choose $f = \chi_{V_1} \tau$, where $V_1 \subset V$ is a subset achieving equality in the definition of the isoperimetric constant, and $\tau$ is a switching function achieving equality in the definition of $\iota^s(V_1)$.
\end{proof}

\begin{remark}
 Theorem \ref{thm:Sobolev-p=1} might at first sight appear to be totally wrong when compared to the standard case \cite[Thm. 1]{ChungYau95}, \cite[Thm. 2]{Ostrovskii}: if the signature is trivial, the right-hand side of the inequality in Theorem \ref{thm:Sobolev-p=1} does not change when a constant is added to the function $f$, but the left-hand side will grow without control.
This apparent contradiction shows an important difference between the classical isoperimetric theory for graphs and the one we are currently developing for their magnetic counterparts:
in Definition \ref{def:isoperimetric constant}, the inequality must hold for \emph{any nonempty subset of $V$}.
If the signature $s$ on the graph $G$ is trivial (or more generally, if the signature is balanced) we will have
$$
 \iota^s(V)+|E(V, \emptyset)| = 0+0=0
$$
and therefore the pair $(G,s)$ \emph{cannot have any isoperimetric dimension}.
This seemingly strange situation stems from a mismatch in the terminology: recall that the classical Cheeger constant corresponds to the $2$-way Cheeger constant of \cite{LLPP}, not the $1$-way one, and our current isoperimetric theory is a refinement of the $1$-way Cheeger constant. 
\end{remark}

\subsection{Sobolev inequality: case $1<p<\infty$}

We follow the strategy of \cite{Ostrovskii} to deduce the case $p>1$ from the case $p=1$.

Let $\alpha \ge 1$. For a complex number $z = re^{i\theta} $ we denote $r^\alpha e^{i\theta} = z \cdot |z|^{\alpha-1}$ by $z^\alpha$.
The following is a complex version of \cite[Lemma 4]{Matousek}.

\begin{lemma}\label{lemma:complex-bernoulli}
For any $z_1,z_2 \in \C$ and $\alpha \ge 1$, we have
$$
\left| z_1^\alpha - z_2^\alpha \right| \le \alpha |z_1 - z_2| \left( |z_1|^{\alpha-1} + |z_2|^{\alpha-1}  \right).
$$
\end{lemma}

\begin{proof}
Without loss of generality, we may assume neither of $z_1, z_2$ is zero. We may also assume $|z_1| \ge |z_2|$, and by rescaling we may take $z_1 = 1$.
Therefore, what we need to show is that for $r \in (0,1]$ and $\theta \in \R$,
$$
\big| 1-r^\alpha e^{i\theta} \big| \le \alpha \big| 1-r e^{i\theta} \big| \cdot \left( 1 + r^{\alpha-1}  \right)
$$
so it suffices to show that $\big| 1-r^\alpha e^{i\theta} \big| \le \alpha \big| 1-r e^{i\theta} \big|$.
Taking squares, this is equivalent to
$$
(1-r^\alpha\cos\theta)^2 + r^{2\alpha}\sin^2\theta \le \alpha^2\big( (1-r\cos\theta)^2 + r^2\sin^2\theta\big),
$$
that is,
$$
1 - 2r^\alpha\cos\theta + r^{2\alpha} \le \alpha^2 \big( 1-2r\cos\theta + r^2\big).
$$
Denoting $x = \cos\theta \in [-1,1]$ and rearranging we get
$$
0 \le \alpha^2 \big( 1-2rx + r^2\big) - (1 - 2r^\alpha x + r^{2\alpha})
$$
The derivative with respect to $x$ of the right-hand side is $2(-\alpha^2r + r^\alpha) \le 0$, so the expression is decreasing in $x$.
Since the inequality is valid when $x=1$ (it reduces to the well-known Bernoulli inequality as in the proof of \cite[Lemma 4]{Matousek}), we achieve the desired conclusion. 
\end{proof}

Following essentially the same argument as in \cite[Sec. 2.2]{Ostrovskii}, we will prove the general Sobolev inequality.

\begin{theorem}\label{thm:Sobolev}
Let $G$ be a finite graph with an  $S_k^1$-valued  signature $s$ and let $1 \le p < \delta$.
Suppose that $(G,s)$ has isoperimetric dimension $\delta$ with isoperimetric constant $c_\delta$.
Let $q$ be defined by $1/p = 1/q + 1/\delta$.
Then there exists a constant $C$ depending only on $p$, $\delta$, $c_\delta$ and $d_\mu$ such that
for any function $f : V \to \C$ we have
$$
\left[  \sum_{u \in V}  \big|  f(u)  \big|^{q}  \mu(u) \right]^{1/q} \le C \left[ \sum_{\{u,v\}\in E}w_{uv}\, \left|f(u)-s_{uv}f(v)\right|^p \right]^{1/p}
$$
\end{theorem}

\begin{proof}
The case $p=1$ was proved in Theorem \ref{thm:Sobolev-p=1}, with $C = 2/c_\delta$. Assume now that $p > 1$.
Let $\alpha \ge 1$.
Applying Theorem \ref{thm:Sobolev-p=1} to the function $f^\alpha$, together with Lemma \ref{lemma:complex-bernoulli} yields
\begin{multline*}
\left[  \sum_{u \in V}  \big|  f(u)  \big|^{\frac{\alpha\delta}{\delta-1}}  \mu(u) \right]^{\frac{\delta-1}{\delta}} \\
 \le \frac{3\alpha}{c_\delta} \sum_{\{u,v\}\in E}w_{uv}\, \left|f(u)-s_{uv}f(v)\right| \big( |f(u)|^{\alpha-1} + |f(v)|^{\alpha-1} \big),
\end{multline*}
where we have used $(s_{uv}f(u))^\alpha = s_{uv} [f(u)]^\alpha$.

Applying H\"older's inequality we obtain, denoting $1/p' = 1-1/p$,
\begin{multline*}
\left[  \sum_{u \in V}  \big|  f(u)  \big|^{\frac{\alpha\delta}{\delta-1}}  \mu(u) \right]^{\frac{\delta-1}{\delta}} \\
 \le \frac{3\alpha}{c_\delta} \left[ \sum_{\{u,v\}\in E}w_{uv}\, \left|f(u)-s_{uv}f(v)\right|^p \right]^{1/p} \left[ \sum_{\{u,v\}\in E}w_{uv}\, \big( |f(u)|^{\alpha-1} + |f(v)|^{\alpha-1} \big)^{p'} \right]^{1/p'}.
\end{multline*}
Since for any $a,b \ge 0$ we have $(a+b)^{p'} \le 2^{p'}a^{p'} + 2^{p'}b^{p'}$, it follows that
\begin{multline*}
\left[ \sum_{\{u,v\}\in E}w_{uv}\, \big( |f(u)|^{\alpha-1} + |f(v)|^{\alpha-1} \big)^{p'} \right]^{1/p'}\\
 \le 2 \left[ \sum_{\{u,v\}\in E}w_{uv}\, \big( |f(u)|^{(\alpha-1)p'} + |f(v)|^{(\alpha-1)p'} \big) \right]^{1/p'}\\
  = 2 \left[ \sum_{u \in v} |f(u)|^{(\alpha-1)p'}d_u \right]^{1/p'} \le 2 d_\mu^{1/p'} \left[ \sum_{u \in v} |f(u)|^{(\alpha-1)p'}\mu(u) \right]^{1/p'}.
\end{multline*}
Set $\alpha = (\delta-1)p/(\delta-p)$, so that $(\alpha-1)p' = \delta p /(\delta-p)= \alpha\delta/(\delta-1)$. Then
\begin{multline*}
\left[  \sum_{u \in V}  \big|  f(u)  \big|^{\frac{\delta p}{\delta-p}}  \mu(u) \right]^{\frac{\delta-1}{\delta}} \\
 \le 2 d_\mu^{1/p'} \frac{(\delta-1)p}{\delta-p}\frac{3}{c_\delta}  \left[ \sum_{\{u,v\}\in E}w_{uv}\, \left|f(u)-s_{uv}f(v)\right|^p \right]^{1/p} \left[ \sum_{u \in v} |f(u)|^{\frac{\delta p}{\delta-p}}\mu(u) \right]^{1/p'}.
\end{multline*}
Noticing that $q = \delta p/(\delta-p)$ and $(\delta-1)/\delta-1/p' = 1/q$, we obtain
\begin{multline*}
\left[  \sum_{u \in V}  \big|  f(u)  \big|^{q}  \mu(u) \right]^{1/q} \\
 \le 2 d_\mu^{1/p'} \frac{(\delta-1)p}{\delta-p}\frac{3}{c_\delta}  \left[ \sum_{\{u,v\}\in E}w_{uv}\, \left|f(u)-s_{uv}f(v)\right|^p \right]^{1/p}. 
 \end{multline*}
\end{proof}

\subsection{Sobolev inequalities with signed Cheeger constants}

As already mentioned, the Cheeger constant from \cite[Def. 3.5]{LLPP} corresponds to the case $\delta=\infty$ in our Definition  \ref{def:isoperimetric constant}.
The following theorem is the corresponding analogue of Theorem \ref{thm:Sobolev-p=1}, and the proof is essentially the same so we omit it. 

\begin{theorem}\label{thm:Sobolev-p=1-Cheeger}
Let $G$ be a finite graph with an  $S_k^1$-valued signature $s$ such that $(G,s)$ has Cheeger constant $h=h_1^s(\mu)$.
For any function $f : V \to \C$ we have
$$
 \sum_{u \in V}  \big|  f(u)  \big| \mu(u) \le \frac{3}{h} \sum_{\{u,v\}\in E}w_{uv}\, \left|f(u)-s_{uv}f(v)\right|.
$$
\end{theorem}

Using the same technique as in the proof of Theorem \ref{thm:Sobolev}, we can then deduce the following.

\begin{theorem}\label{thm:Sobolev-Cheeger}
Let $G$ be a finite graph with an  $S_k^1$-valued  signature $s$ such that $(G,s)$ has Cheeger constant $h=h_1^s(\mu)$, and let $p \ge 1$.
For any function $f : V \to \C$ we have
$$
\left[ \sum_{u \in V}  \big|  f(u)  \big|^p \mu(u) \right]^{1/p} \le \frac{6p d_\mu^{1/p'}}{h} \left[ \sum_{\{u,v\}\in E}w_{uv}\, \left|f(u)-s_{uv}f(v)\right|^p \right]^{1/p}.
$$
\end{theorem}

\begin{remark}
The case $p=2$ of Theorem \ref{thm:Sobolev-Cheeger} is essentially contained in the proof of \cite[Lemma 4.4]{LLPP}, though our factor of 6 is slightly worse than theirs.
\end{remark}

The constant in Theorem \ref{thm:Sobolev-p=1-Cheeger} is essentially optimal, as shown by the following result. Compare to \cite[Cor. 1]{Chung-Tetali}.

\begin{corollary}\label{cor:variational-characterization-Cheeger}
Let $G$ be a finite graph with an  $S_k^1$-valued signature $s$ such that $(G,s)$ has Cheeger constant $h=h_1^s(\mu)$.
Then
$$
\frac{h}{3} \le \inf_{f \not= 0} \left[ \frac{ \sum_{\{u,v\}\in E}w_{uv}\, \left|f(u)-s_{uv}f(v)\right| }{ \sum_{u \in V}  \big|  f(u)  \big| \mu(u) } \right] \le h
$$
\end{corollary}

\begin{proof}
The lower bound follows immediately from Theorem \ref{thm:Sobolev-p=1-Cheeger}, whereas to prove the upper bound it suffices to choose $f = \chi_{V_1} \tau$, where $V_1 \subset V$ is a subset achieving equality in the definition of the Cheeger constant, and $\tau$ is a switching function achieving equality in the definition of $\iota^s(V_1)$.
In fact, for this function the quotient of interest equals $h$.
\end{proof}

\subsection{What about the balanced case?}

Recall that in the balanced case, the $1$-way Cheeger constant will be zero.
It is thus reasonable to hope that the $2$-way Cheeger constant might yield Sobolev inequalities.
But for a balanced signature this is just the classical Cheeger constant, and we cannot hope to do better than what is already known in the case without signatures.

Suppose that $G$ is a finite graph with balanced $S^1_k$-valued signature $s$. Therefore, there exists a switching function $\tau : V \to S_k^1$ such that
 whenever $u \sim v$ we have ${\tau(u)^{-1} s_{uv}\tau(v) = 1}$. Thus, for any $f : V \to \C$ and $u \sim v$
\begin{multline*}
\left|f(u)-s_{uv}f(v)\right| \\
= \left|\tau(u)^{-1} f(u)- \tau(u)^{-1}s_{uv}f(v)\right| = \left|\tau(u)^{-1} f(u)- \tau(v)^{-1}f(v)\right|
\end{multline*}
and thus we cannot expect to do any better than the classical theory already does for the function $\tau^{-1} f$, namely:
if $1 \le p < \delta$, the graph $G$ has isoperimetric dimension $\delta$ with constant $c_\delta$ (in the sense of \cite{ChungYau95,Ostrovskii}), and if
we define $q$ by $1/p = 1/q + 1/\delta$,
then there exists a constant $C$ depending only on $p$, $\delta$, $c_\delta$ and $d_\mu$ such that
for any function $f : V \to \C$ we have
$$
\inf_{z\in \C} \left[  \sum_{u \in V}  \big|  f(u) - z \tau(u)  \big|^{q}  \mu(u) \right]^{1/q} \le C \left[ \sum_{\{u,v\}\in E}w_{uv}\, \left|f(u)-s_{uv}f(v)\right|^p \right]^{1/p}.
$$

\section{Signed Cheeger constants for cartesian products}\label{sec:products}

For two weighted graphs with vertex measures $G_1 = (V_1,E_1,w^1, \mu_1)$ and $G_2 = (V_2,E_2,w^2, \mu_2)$ with corresponding $\Gamma$-valued signatures $s^1$ and $s^2$, we define their signed Cartesian product $G_1 \times G_2$ as the following graph:
the vertex set is $V= V_1 \times V_2$, the edge set $E$ determined by
$$
(u,v) \sim (u',v') \Leftrightarrow \big[ \text{ $u=u'$ and $v \sim v'$, or $v=v'$ and $u \sim u'$ } \big],
$$
the product weight $w$ is defined by
$$
w_{(u,v)(u',v')} = \begin{cases}
w^2_{vv'}\mu_1(u) &\text{if $u=u'$ and $v \sim v'$},\\
w^1_{uu'}\mu_2(v) &\text{if $v=v'$ and $u \sim u'$},
\end{cases}
$$
the vertex measure is the product measure $\mu = \mu_1 \times \mu_2$,
and the product signature $s = s^1 \times s^2$ is given by
$$
s_{(u,v)(u',v')} = \begin{cases}
s^2_{vv'} &\text{if $u=u'$ and $v \sim v'$},\\
s^1_{uu'} &\text{if $v=v'$ and $u \sim u'$}.
\end{cases}
$$
Note that this is the standard definition of the Cartesian product of graphs, and we just have added the natural choice for a product signature.

In contrast with the classical case \cite[Thms. 2 and 3]{Chung-Tetali},
the signed Cheeger constant behaves additively with respect to this cartesian product.

\begin{theorem}\label{thm:Cheeger-constant-of-product}
Let $(G_1,s^1), \dotsc, (G_m,s^m)$ be weighted graphs with $S^1_k$-valued signatures and corresponding vertex measures $\mu_1, \dotsc, \mu_m$.
For convenience we write $h(G_j)=h_1^{s^j}(\mu_j)$ and $h(G_1 \times \cdots \times G_k) = h_1^{s^1 \times \cdots \times s^k }( \mu_1 \times \cdots \times \mu_k )$.
Then
$$
 \frac{1}{3} \sum_{j=1}^m h(G_j) \le h( G_1 \times \cdots \times G_m ) \le 3 \sum_{j=1}^m h(G_j).
$$
\end{theorem}

\begin{proof}
From the proof of Corollary \ref{cor:variational-characterization-Cheeger}, for each $j=1, \dotsc, m$  there is a nonzero function $f_j : V_j \to \C$ achieving
$$
\frac{ \sum_{\{u_j,u_j'\}\in E_j}w^j_{u_ju_j'}\, \left|f_j(u_j)-s^j_{u_ju_j'}f_j(u_j')\right| }{ \sum_{u_j \in V_j}  \big|  f_j(u)  \big| \mu_j(u_j) } = h(G_j).
$$
Let us define $F : V_1 \times \cdots V_m \to \C$ by $F(v_1, \dotsc, v_m) = f_1(v_1) \cdots f_m(v_m)$.
Then, using Corollary \ref{cor:variational-characterization-Cheeger}
\begin{multline*}
\frac{1}{3}h( G_1 \times \cdots \times G_m ) \le \\ 
\frac{ \sum_{ \{(u_1,\dotsc,u_m), (u_1',\dotsc,u_m')\} \in E } w_{(u_1,\dotsc,u_m), (u_1',\dotsc,u_m')} \left|F(u_1,\dotsc,u_m)-s_{(u_1,\dotsc,u_m), (u_1',\dotsc,u_m')}F(u_1',\dotsc,u_m')\right|   }{ \sum_{ (u_1,\dotsc,u_m) \in V}  \big|  F(u_1,\dotsc,u_m)  \big| \mu(u_1,\dotsc,u_m)  } \\
= \frac{ A_1 + \cdots + A_m  }{ \prod_{j=1}^m \left( \sum_{u_j \in V_j} |f_j(u_j)| \mu_j(u_j) \right) }
\end{multline*}
where for each $j=1, \dotsc, m$,
\begin{multline*}
A_j = \prod_{i \not= j} \bigg( \sum_{u_i \in V_i} |f_i(u_i)| \mu_i(u_i) \bigg)  \sum_{\{u_j,u_j'\}\in E_j}w^j_{u_ju_j'}\, \left|f_j(u_j)-s^j_{u_ju_j'}f_j(u_j')\right| \\
= h(G_j) \prod_{i= 1}^m \bigg( \sum_{u_i \in V_i} |f_i(u_i)| \mu_i(u_i) \bigg),
\end{multline*}
and thus we conclude
$$
\frac{1}{3}h( G_1 \times \cdots \times G_m ) \le  \sum_{j=1}^m h(G_j)
$$

From the proof of Corollary \ref{cor:variational-characterization-Cheeger},  there is a nonzero function ${f : V_1 \times \cdots \times V_m \to \C}$ whose corresponding quotient achieves the Cheeger constant for ${G_1 \times \cdots \times G_m}$,
so we can write
\begin{multline*}
h(G_1 \times \cdots \times G_m) = \\
\frac{ \sum_{ \{(u_1,\dotsc,u_m), (u_1',\dotsc,u_m')\} \in E } w_{(u_1,\dotsc,u_m), (u_1',\dotsc,u_m')} \left|f(u_1,\dotsc,u_m)-s_{(u_1,\dotsc,u_m), (u_1',\dotsc,u_m')}f(u_1',\dotsc,u_m')\right|   }{ \sum_{ (u_1,\dotsc,u_m) \in V}  \big|  f(u_1,\dotsc,u_m)  \big| \mu(u_1,\dotsc,u_m)  } \\
= \frac{ B_1 + \cdots + B_m  }{ \sum_{ (u_1,\dotsc,u_m) \in V}  \big|  f(u_1,\dotsc,u_m)  \big| \mu(u_1,\dotsc,u_m)  } 
\end{multline*}
with, for each $j=1, \dotsc, m$,
\begin{multline*}
B_j = \\
\sum_{u_i \in V_i, i \not=j} \bigg( \prod_{k \not= j} \mu_k(u_k) \bigg) \cdot\\
 \sum_{\{u_j,u_j'\}\in E_j}w^j_{u_ju_j'}\, \left|f(u_1, \dotsc, u_{j-1}, u_j, u_{j+1}, \dotsc,u_m)-s^j_{u_ju_j'}f(u_1, \dotsc, u_{j-1}, u_j', u_{j+1}, \dotsc,u_m)\right| \\
\ge  \sum_{u_i \in V_i, i \not=j}  \bigg( \prod_{k \not= j} \mu_k(u_k) \bigg) \frac{h(G_j)}{3} \sum_{u_j \in V_j} |f(u_1, \dotsc, u_m)| \mu_j(u_j),
\end{multline*}
and where we have used  Theorem \ref{thm:Sobolev-p=1-Cheeger} to obtain the inequality above.
It then follows that
$$
h(G_1 \times \cdots \times G_m) \ge \frac{1}{3} \sum_{j=1}^m h(G_j).
$$
\end{proof}

The paper \cite{LLPP}, where the signed Cheeger constant was defined, did not provide explicit examples of calculations of such constants.
We will use Theorem \ref{thm:Cheeger-constant-of-product} to evaluate signed Cheeger constants for discrete tori, but first we will need to do the simple calculation for cycles.

\begin{proposition}\label{prop:Cheeger-constant-of-cycle}
Let $\mathcal{C}=(V,E)$ be a graph which is a cycle with $n$ edges each having weight $1$, endowed with the vertex measure $\mu \equiv 1$, and let $s$ be an $S^1_k$-valued (or $S^1$-valued) signature for $\mathcal{C}$. Then $h^s(\mathcal{C}) = |1-s(\mathcal{C})|/n$.
\end{proposition}

\begin{proof}
Let $X$ be a nonempty subset of $V$.
If $X=V$, then using Proposition \ref{prop:frustration-of-cycle},
$$
\frac{\iota^s(X) + |E(X,X^c)|}{\vol_\mu(X)} =  \frac{ \iota^s(\mathcal{C}) }{ n } = \frac{ |1-s(\mathcal{C})| }{n}.
$$
When $\emptyset \not= X \subsetneq V$, note that $\iota^s(X) = 0$ since the graph induced by $X$ is a disjoint union of paths. Therefore,
$$
\frac{\iota^s(X) + |E(X,X^c)|}{\vol_\mu(X)} =  \frac{ |E(X,X^c)| }{ |X| } \ge \frac{2}{n-1}.
$$
Since
$$
\frac{2}{n-1} \ge\frac{ |1-s(\mathcal{C})| }{n},
$$
the result follows.
\end{proof}

\begin{remark}\label{rem:isoperimetric-constant-of-cycle}
The same argument shows that for any $\delta>1$, such a cycle $\mathcal{C}$ has isoperimetric dimension $\delta$ with isoperimetric constant ${c_\delta = |1-s(\mathcal{C})|/n^{(\delta-1)/\delta}}$.
\end{remark}

We can now easily calculate signed Cheeger constants for discrete tori, simply by putting together Theorem \ref{thm:Cheeger-constant-of-product} and Proposition \ref{prop:Cheeger-constant-of-cycle}.

\begin{corollary}
Let $\mathcal{C}_1, \dotsc, \mathcal{C}_m$ be cycles of respective lengths $n_1, \dotsc, n_m$, endowed with respective $S^1_k$-valued signatures $s^1, \dotsc, s^m$, and with edge weights and vertex measures that are all equal to $1$. Then
$$
 \frac{1}{3} \sum_{j=1}^m \frac{|1-s^j(\mathcal{C}_j)|}{n_j} \le h( \mathcal{C}_1 \times \cdots \times \mathcal{C}_m ) \le 3 \sum_{j=1}^m \frac{|1-s^j(\mathcal{C}_j)|}{n_j}.
$$
\end{corollary}

\section{The heat kernel: eigenvalue estimates for discrete magnetic laplacians involving Sobolev constants}\label{sec:heat}

Our next goal is to give bounds for the eigenvalues of the signed Laplacian depending on Sobolev constants,
analogous to the ones in \cite{ChungYau95}.
Let us start with some notation.

We consider the following (normalized) magnetic Laplacian $\lap_{\sigma,\mu}$ associated to the
weighted graph $(G,w)$ with signature $\sigma: E^{or} \to \Gamma$ and
vertex measure $\mu: V \to {\mathbb R}^+$. For any function $f: V\to
\mathbb{C}$, and any vertex $u\in V$, we define
\begin{equation} \label{eq:Deltasmu}
\lap_{\sigma,\mu}f(u):=\frac{1}{\sqrt{\mu(u)}}\sum_{v,v\sim u}w_{uv}\left(\frac{f(u)}{\sqrt{\mu(u)}}-\sigma_{uv}\frac{f(v)}{\sqrt{\mu(v)}}\right).
\end{equation}
Note that the summation in \eqref{eq:Deltasmu} over the vertices $v$
adjacent to $u$ can also be understood as a summation over the
oriented edges $e=(u,v) \in E^{or}$, and the signature is evaluated at
$(u,v)$.
To simplify the notation, we will often drop the dependance on $\mu$ from the notation and write $\lap_\sigma$ instead of $\lap_{\sigma,\mu}$.
If we simply write $\lap$, we are referring to the Laplacian corresponding to the trivial signature.

The Laplacian $\lap_{\sigma,\mu}$ has the following decomposition
$$
\lap_{\sigma,\mu}=(D_{\mu})^{-1/2}(D-A^s)(D_{\mu})^{-1/2},
$$
where $D$ and $D_{\mu}$ are the diagonal matrices with $D_{uu}=d_u$
and $(D_{\mu})_{uu}=\mu(u)$ for all $u\in V$, while $A^s$ is the
weighted signed adjacency matrix with
$$
A^s_{uv} := \begin{cases}
			0,				& \text{$u=v$ or $\{u,v\}\not\in E$},\\
			w_{uv}s_{uv},	& \text{$\{u,v\}\in E$.}
\end{cases}
$$
Note that our Laplacian is not the same one that was considered in \cite{LLPP}, since we have normalized it differently:
they considered $(D_{\mu})^{-1}(D-A^s)$ instead,
but it is easy to see that the eigenvalues are the same in our case and therefore the Cheeger inequalities from \cite{LLPP} are still valid. 
By \eqref{eq:signatureKEY}, the matrix $\lap_{\sigma,\mu}$ is Hermitian, and
hence all its eigenvalues are real which can be listed with
multiplicity as follows: 
\begin{equation}
  0\leq \lambda_1(\lap_{\sigma,\mu})\leq \lambda_2(\lap_{\sigma,\mu})\leq\cdots\leq\lambda_N(\lap_{\sigma,\mu})\leq 2d_{\mu}.
\end{equation}
For notational simplicity, we will often just write $\lambda_j$ instead of $\lambda_j(\lap_{\sigma,\mu})$.

\begin{definition}\label{defn:heat kernel}
We can express the Laplacian as
$$
\lap_\sigma = \sum_{j=1}^N \lambda_j P_j
$$
where $P_j$ is the orthogonal projection onto the span of the $j$-th normalized eigenfunction $\gamma_j$.
The \emph{heat kernel} $K^\sigma_t$ of $(G,s)$ is defined to be the $N \times N$ matrix
$$
K^\sigma_t = \sum_j e^{-\lambda_j t}P_j = e^{-t\lap_\sigma}
$$
The heat kernel corresponding to $\lap$ will be denoted by $K_t$.
\end{definition}

Some basic properties of the heat kernel, which follow directly from its definition, are stated below.

\begin{lemma}\label{lemma:properties heat kernel}
For $u, v \in V$ we have
\begin{enumerate}[(i)]
\item $K^\sigma_t$ is Hermitian, so $K^\sigma_t(u,v) = \overline{K^\sigma_t(v,u)}$.
\item\label{lemma:PHK:composition}
 For any $0 \le a \le t$,
$$
K^\sigma_t(u,v) = \sum_{z \in V} K^\sigma_a(u,z)K^\sigma_{t-a}(z,v)
$$
\item For $f : V \to \C$,
$$
K^\sigma_tf(u) = \sum_{v \in V} K^\sigma_t(u,v) f(v).
$$
\item\label{lemma:PHK:heat-eqn}
 $K^\sigma_t$ satisfies the heat equation
$$
\frac{d}{dt} K^\sigma_t = - \lap_\sigma K^\sigma_t 
$$
\item\label{lemma:PHK:unsigned}
 $K_t(u,v) \ge 0$ and $K_t \sqrt{\mu} = \sqrt{\mu}$ for any $t \ge 0$.
\end{enumerate}
\end{lemma}

The following is an adaptation of \cite[Lemma 1.1]{Dodziuk-Mathai} to our setting.

\begin{lemma}\label{lemma:pos-preserving}
Let $f : V \to \R$ and $\lambda > 0$. If $(\lap + \lambda I) f \ge 0$, then $f \ge 0$ (that is, the operator $(\lap +\lambda I)^{-1}$ is positivity-preserving).
\end{lemma}

\begin{proof}
Suppose $(\lap + \lambda I) f \ge 0$ but $f$ attains a negative value.
Let $u \in V$ be the vertex where $\frac{1}{\sqrt{\mu}}f$ attains its minimum value. Now, by definition of the Laplacian,
$$
(\lap + \lambda I) f (u) = \lambda f(u) +  \frac{1}{\sqrt{\mu(u)}}\sum_{v,v\sim u}w_{uv}\left( \frac{f(u)}{\sqrt{\mu(u)}} - \frac{f(v)}{\sqrt{\mu(v)}} \right).
$$
The left-hand side is nonnegative by assumption, but the right-hand side is strictly negative, reaching a contradiction.
\end{proof}

Now we can get a weighted version of Kato's inequality for the magnetic Laplacian (compare to \cite[Lemma 1.2]{Dodziuk-Mathai}).

\begin{lemma}\label{lemma:Kato}
For every $f : V \to \C$, one has the pointwise inequality
$$
|f| \cdot \lap|f| \le \Re\big( \lap_\sigma f \cdot \overline{f} \big).
$$
\end{lemma}

\begin{proof}
Let $u \in V$. From the definitions of the Laplacians,
\begin{align*}
(|f| \cdot \lap|f|)(u) &=  \frac{1}{\sqrt{\mu(u)}}\sum_{v,v\sim u}w_{uv}\left(\frac{|f(u)|^2}{\sqrt{\mu(u)}}-\frac{|f(u)|\cdot|f(v)|}{\sqrt{\mu(v)}}\right) \\
\big( \lap_\sigma f \cdot \overline{f} \big)(u) &= \frac{1}{\sqrt{\mu(u)}}\sum_{v,v\sim u}w_{uv}\left(\frac{ |f(u)|^2  }{\sqrt{\mu(u)}}-\sigma_{uv}\frac{ \overline{f(u)} f(v)}{\sqrt{\mu(v)}}\right)
\end{align*}
and therefore
\begin{multline*}
(|f| \cdot \lap|f|)(u) - \Re\big( \lap_\sigma f \cdot \overline{f} \big)(u) = \\
\frac{1}{\sqrt{\mu(u) }}\sum_{v,v\sim u} \frac{w_{uv} }{ \sqrt{\mu(v)}} \Re\big[ \sigma_{uv} \overline{f(u)} f(v) - |f(u)| \cdot |f(v)| \big] \le 0.
\end{multline*}
\end{proof}

As a consequence we get a domination result between the heat kernels, a weighted version of \cite[Thm. 1.5]{Dodziuk-Mathai}.

\begin{theorem}\label{thm:heat-kernel-domination}
For any signature $\sigma$ and any $f : V \to \C$, we have the pointwise inequality
$$
| e^{-t \lap_\sigma} f | \le e^{-t \lap } |f|.
$$
As a consequence, for any $f,g : V \to C$ we have
$$
\big| \pair{ e^{-t \lap_\sigma} f }{ g} \big| \le \pair{ e^{-t \lap_\sigma} |f| }{ |g|}
$$
and for any $u,v \in V$ we have
$$
|K^\sigma_t(u,v)| \le K_t(u,v)
$$
\end{theorem}

\begin{proof}
It follows from Lemma \ref{lemma:Kato} that for any $\lambda>0$ and any $g : V \to \C$, we have the pointwise inequality
$$
|g| \cdot (\lap+\lambda I)|g| \le \Re \big[ (\lap_\sigma + \lambda I)g \cdot \overline{g} \big] \le \big|(\lap_\sigma + \lambda I)g\big| \cdot |g|,
$$
and therefore, at all vertices where $g \not=0$,
$$
(\lap+\lambda I)|g| \le \big|(\lap_\sigma + \lambda I)g\big|.
$$
At a vertex where $g(u)=0$ the left-hand side of the inequality is nonpositive whereas the right-hand side is nonnegative, so
the above inequality in fact holds in general.
Now let $g = (\lap\sigma + \lambda I)^{-1} f$, so that
$$
(\lap+\lambda I) \big |(\lap_\sigma + \lambda I)^{-1} f \big| \le |f|.
$$
By Lemma \ref{lemma:pos-preserving} the operator $(\lap + \lambda I)^{-1}$ is positivity preserving, and therefore we have, pointwise,
$$
\big |(\lap_\sigma + \lambda I)^{-1} f \big| \le (\lap+\lambda I)^{-1}|f|.
$$
By induction, for any $n\in\N$
$$
\big |(\lap_\sigma + \lambda I)^{-n} f \big| \le (\lap+\lambda I)^{-n}|f|.
$$
Since $e^{-tA} = \lim_{n\to\infty} (I + (t/n) A)^{-n}$, we obtain the first inequality.
The second inequality in the Theorem follows from the first and the definition of the inner product, whereas the inequality for the heat kernels follows from taking $f=\delta_v$, $g = \delta_u$.
\end{proof}

We are now ready to prove the main theorem of this section.

\begin{theorem}\label{thm:eigenvalue-bounds}
Suppose that $(G,\sigma)$ has isoperimetric dimension $\delta>2$ with constant $c_\delta$.
Then the eigenvalues of the Laplacian $\lap_\sigma$ satisfy, for any $t>0$,
$$
\sum_{j=1}^N e^{-\lambda_jt} \le C_\delta \frac{ \vol_\mu(G)}{t^\delta/2}
$$
where $C_\delta$ is a constant depending only on $\delta$, $c_\delta$ and $d_\mu$.
\end{theorem}

\begin{proof}
Our general strategy is similar to that of \cite[Sec. 4]{ChungYau95}.
We start by calculating the derivative of a diagonal term of the heat kernel, using Lemma \ref{lemma:properties heat kernel}.\eqref{lemma:PHK:composition} and \eqref{lemma:PHK:heat-eqn}.
\begin{align*}
\frac{d}{dt} K^\sigma_t(u,u) &= \frac{d}{dt} \sum_{v\in V}  K^\sigma_{t/2}(u,v) K^\sigma_{t/2}(v,u) \\
&= \sum_{v\in V} \left[  \frac{d}{dt}K^\sigma_{t/2}(u,v) K^\sigma_{t/2}(v,u) + K^\sigma_{t/2}(u,v)\frac{d}{dt}  K^\sigma_{t/2}(v,u) \right] \\
&= - \frac{1}{2} \sum_{v\in V} \left[  \lap_\sigma K^\sigma_{t/2}(u,v) K^\sigma_{t/2}(v,u) + K^\sigma_{t/2}(u,v) \lap_\sigma K^\sigma_{t/2}(v,u) \right].
\end{align*}
For simplicity, let us calculate separately
\begin{multline*}
\sum_{v\in V} K^\sigma_{t/2}(u,v) \lap_\sigma K^\sigma_{t/2}(v,u) \\
= \sum_{v\in V} K^\sigma_{t/2}(u,v) \frac{1}{\sqrt{\mu(v)}} \sum_{z,z\sim v} w_{vz} \left( \frac{K^\sigma_{t/2}(v,u)}{\sqrt{\mu(v)}}-\sigma_{vz}\frac{K^\sigma_{t/2}(z,u)}{\sqrt{\mu(z)}}\right) \\
= \sum_{v \sim z} w_{vz} \left( \frac{K^\sigma_{t/2}(v,u)}{\sqrt{\mu(v)}}-\sigma_{vz}\frac{K^\sigma_{t/2}(z,u)}{\sqrt{\mu(z)}}\right) \left( \frac{K^\sigma_{t/2}(u,v)}{\sqrt{\mu(v)}}-\sigma_{zv}\frac{K^\sigma_{t/2}(u,z)}{\sqrt{\mu(z)}}\right) \\
= \sum_{v \sim z}  w_{vz} \left| \frac{K^\sigma_{t/2}(v,u)}{\sqrt{\mu(v)}}-\sigma_{vz}\frac{K^\sigma_{t/2}(z,u)}{\sqrt{\mu(z)}}\right|^2
\end{multline*}
Since $\lap_\sigma K^\sigma_{t/2}(u,v) K_{t/2}(v,u) = \overline{ K^\sigma_{t/2}(u,v) \lap_\sigma K^\sigma_{t/2}(v,u)}$, we conclude
$$
\frac{d}{dt} K^\sigma_t(u,u) = - \sum_{v \sim z} w_{vz} \left| \frac{K^\sigma_{t/2}(v,u)}{\sqrt{\mu(v)}}-\sigma_{vz}\frac{K^\sigma_{t/2}(z,u)}{\sqrt{\mu(z)}}\right|^2
$$
By Theorem \ref{thm:Sobolev}, with $q = 2\delta/(\delta-2)$,
$$
\frac{d}{dt} K^\sigma_t(u,u) \le - \frac{c_\delta^2}{144d_\mu} \left( \frac{\delta-2}{\delta-1} \right)^2 \left[  \sum_{v \in V} \bigg| \frac{K^\sigma_{t/2}(v,u)}{\sqrt{\mu(v)}} \bigg|^q \mu(v)  \right]^{2/q} 
$$
Now, by H\"older's inequality with conjugate indices $r := q-1$ and ${r' = (q-1)/(q-2)}$ applied to the functions
$$
\bigg| \frac{K^\sigma_{t/2}(v,u)}{\sqrt{\mu(v)}} \bigg|^{\frac{q}{q-1}}, \qquad \bigg| \frac{K^\sigma_{t/2}(v,u)}{\sqrt{\mu(v)}} \bigg|^{\frac{q-2}{q-1}}
$$
we obtain
\begin{multline*}
\sum_{v \in V} \bigg| \frac{K^\sigma_{t/2}(v,u)}{\sqrt{\mu(v)}} \bigg|^2 \mu(v) \\
\le \left[  \sum_{v \in V} \bigg| \frac{K^\sigma_{t/2}(v,u)}{\sqrt{\mu(v)}} \bigg|^q \mu(v)  \right]^{\frac{1}{q-1}}  \left[  \sum_{v \in V} \bigg| \frac{K^\sigma_{t/2}(v,u)}{\sqrt{\mu(v)}} \bigg| \mu(v)  \right]^{\frac{q-2}{q-1}}
\end{multline*}

Observe that, using Lemma \ref{lemma:properties heat kernel}.\eqref{lemma:PHK:composition} again
$$
\sum_{v \in V} \bigg| \frac{K^\sigma_{t/2}(v,u)}{\sqrt{\mu(v)}} \bigg|^2 \mu(v) = \sum_{v \in V} \big| K^\sigma_{t/2}(v,u)  \big|^2  = \sum_{v \in V}  K^\sigma_{t/2}(v,u)K^\sigma_{t/2}(u,v) = K^\sigma_t(u,u),
$$
(so in particular $K^\sigma_t(u,u) \ge 0$),
whereas from Theorem \ref{thm:heat-kernel-domination} and Lemma \ref{lemma:properties heat kernel}.\eqref{lemma:PHK:unsigned} we have
$$
\sum_{v \in V} \bigg| \frac{K^\sigma_{t/2}(v,u)}{\sqrt{\mu(v)}} \bigg| \mu(v) = \sum_{v \in V} \big| K^\sigma_{t/2}(v,u) \big| \sqrt{\mu(v)} \le \sum_{v \in V}  K_{t/2}(v,u) \sqrt{\mu(v)} = \sqrt{\mu(u)} .
$$
Therefore,
$$
\frac{d}{dt} K^\sigma_t(u,u) \le - \frac{c_\delta^2}{144d_\mu}  \left( \frac{\delta-2}{\delta-1} \right)^2 K^\sigma_t(u,u)^{2(q-1)/q} \mu(u)^{-(q-2)/q}.
$$
Using the fact that $1-2(q-1)/q = -2/\delta$, we now consider
\begin{align*}
\frac{d}{dt} \left[  K^\sigma_t(u,u)^{1-2(q-1)/q}  \right] &= -\frac{2}{\delta} K^\sigma_t(u,u)^{-2(q-1)/q} \frac{d}{dt} K^\sigma_t(u,u) \\
&\ge \frac{c_\delta^2}{72 \delta d_\mu}  \left( \frac{\delta-2}{\delta-1} \right)^2  \mu(u)^{-(q-2)/q}.
\end{align*}
Integrating,
\begin{align*}
K^\sigma_t(u,u)^{-2/\delta} &\ge \frac{c_\delta^2}{72 \delta d_\mu}  \left( \frac{\delta-2}{\delta-1} \right)^2  \mu(u)^{-(q-2)/q}t + 1\\
&\ge \frac{c_\delta^2}{72 \delta d_\mu}  \left( \frac{\delta-2}{\delta-1} \right)^2  \mu(u)^{-(q-2)/q}t 
\end{align*}
\end{proof}
That is,
$$
K^\sigma_t(u,u) \le C_\delta \frac{ \mu(u) }{ t^{\delta/2}},  \quad \text{where} \quad C_\delta =  \frac{(72 \delta d_\mu)^{\delta/2}}{c_\delta^\delta}  \left( \frac{\delta-1}{\delta-2} \right)^\delta.
$$
Adding up over all $u \in V$, we conclude
$$
\sum_{j=1}^N e^{-\lambda_j t} = \sum_{u \in V} K^\sigma_t(u,u) \le C_\delta \frac{ \vol_\mu(G) }{ t^{\delta/2}}.
$$

As a consequence we give lower bounds for the eigenvalues of the discrete magnetic Laplacian,
which are analogous to Polya's conjecture for Dirichlet eigenvalues of regular domains in $\R^n$ \cite{Polya-Szego}.

\begin{corollary}
Suppose that $(G,\sigma)$ has isoperimetric dimension $\delta>2$ with constant $c_\delta$.
Then the $k$-th eigenvalue $\lambda_k$ of the discrete magnetic Laplacian $\lap_\sigma$ satisfies
$$
\lambda_k \ge C'_\delta \frac{ k }{ \vol_\mu(G) }
$$
where $C'_\delta$ is a constant depending only on $\delta$, $c_\delta$ and $d_\mu$.
\end{corollary}

\begin{proof}
It follows from Theorem \ref{thm:eigenvalue-bounds} that
$$
k e^{-\lambda_k t} \le C_\delta \frac{ \vol_\mu(G) }{ t^{\delta/2}} \quad \text{ and thus } \quad k  \le C_\delta \vol_\mu(G) \frac{  e^{\lambda_k t} }{ t^{\delta/2}}.
$$
The function $e^{\lambda_kt}/t^{\delta/2}$ is minimized when $t = \delta/2\lambda_k$, and thus
$$
k \le C_\delta \vol_\mu(G) \bigg( \frac{2\lambda_ke}{\delta} \bigg)^{\delta/2}
$$
which implies
$$
\lambda_k \ge \frac{\delta}{2e} \bigg( \frac{k}{ C_\delta \vol_\mu(G) } \bigg)^{2/\delta} = C_\delta'  \bigg( \frac{k}{ \vol_\mu(G) } \bigg)^{2/\delta} \quad \text{with} \quad C_\delta' = \frac{\delta}{2e C_\delta^{2/\delta}}.
$$
\end{proof}

\section{The $S^1$-valued case}\label{sec:S1-valued}

All of our Sobolev inequalities have versions for the case where the signature is $S^1$-valued.
We start with a version of Lemma \ref{lemma:key}, which in turn is adapted from \cite[Lemma 4.7]{LLPP}.

For any $t\in (0,1]$, we define $X_{t}: \overline{B_1(0)}\to
\C$ by
$$
  X_t(z):=
  \begin{cases}
  z/|z|, & \text{if } z\in \overline{B_1(0)}\setminus B_{t}(0),\\
  0 & \text{if } z \in B_t(0).
\end{cases}
$$

\begin{lemma}\label{lemma:keyS1}
  For any two points $z_1,z_2\in \overline{B_1(0)}$, we have
  \begin{equation}\label{eq:keyS1}
    \int_0^1\left|X_{t}(z_1)-X_{t}(z_2)\right|\, dt\leq 2|z_1-z_2|.
  \end{equation}
\end{lemma}

\begin{proof}
  Without loss of generality, we may assume that $|z_1|\ge|z_2|>0$. Observe that
  \begin{multline*}
    \int_0^1\left|X_{t}(z_1)-X_{t}(z_2)\right|\, dt
    \leq \left|\frac{z_1}{|z_1|}-\frac{z_2}{|z_2|}\right||z_2|+\big(|z_1|-|z_2|\big) \\
=     \left|\frac{z_1}{|z_1|}|z_2|- z_2 \right|+\big(|z_1|-|z_2|\big) \le 2 | z_1 - z_2 |.
  \end{multline*}
\end{proof}

With Lemma \ref{lemma:keyS1} in hand, we can now follow exactly the same strategies of proof to get versions of Theorems \ref{thm:Sobolev-p=1}, \ref{thm:Sobolev}, \ref{thm:Sobolev-p=1-Cheeger} and
\ref{thm:Sobolev-Cheeger},
and of Corollaries \ref{cor:variational-characterization-isoperimetric-constant} and \ref{cor:variational-characterization-Cheeger}
for $S^1$-valued signatures: the only change comes from the factor of $3$ coming from Lemma \ref{lemma:key} being improved to a $2$ thanks to Lemma \ref{lemma:keyS1}.

\section{Relationship to Balian-Low type theorems in finite dimensions}\label{sec:BL}

The original inspiration for this paper was to provide a possible method for proving the finite Balian-Low conjecture from \cite{Lammers-Stampe}.
The conjecture has recently been proved by Nitzan and Olsen \cite{Nitzan-Olsen-finite-BL}, based on methods previously developed by the same authors in \cite{Nitzan-Olsen}.
Though it is possible to give a proof for the conjecture using the language we have developed here of Sobolev inequalities, the proof is essentially the same as that of \cite{Nitzan-Olsen-finite-BL} and therefore we will only briefly sketch how their proof can be interpreted in our language.

By applying the finite Zak transform \cite[Defn. 2.3]{Nitzan-Olsen-finite-BL}, the expression appearing in the Balian-Low conjecture looks very much like one side of a classical Sobolev inequality on the discrete torus, except that certain complex numbers of modulus one make an appearance (due to the properties of the Zak transform as in \cite[Lemma 2.4.(i)]{Nitzan-Olsen-finite-BL}).
Thus, this expression can be interpreted as one half of a Sobolev inequality on a \emph{magnetic} discrete torus.
To finish the proof, we would then need lower bounds for an isoperimetric constant for the aforementioned magnetic discrete torus.
This is provided by \cite[Lemma 3.1]{Nitzan-Olsen-finite-BL}, which is also a fundamental step in the proof of Nitzan and Olsen.

\section*{Acknowledgements}

The author thanks Profs. William B. Johnson and Keri Kornelson for helpful discussions on the subject of this paper.

\bibliography{sob-references}
\bibliographystyle{amsalpha}

\end{document}